\documentclass{amsart}

\usepackage{float}
\usepackage{graphicx}
\usepackage{tikz}
\usetikzlibrary{trees}
\usepackage{hyperref}
\usepackage{amsmath}

\newtheorem{theorem}{Theorem}[section]
\newtheorem*{theorem*}{Theorem}
\newtheorem*{acknowledgements*}{Acknowledgement}

\newtheorem{claim}[theorem]{Claim}

\theoremstyle{definition}
\newtheorem{definition}[theorem]{Definition}

\theoremstyle{remark}

\numberwithin{equation}{section}

\begin{document}

\title{A forest of Eisensteinian triplets}

\author{Noam Zimhoni}

\email{noam.zimhoni@mail.huji.ac.il}


\date{March 3rd 2019}


\begin{abstract}
In 1934 B. Berggren first discovered the surprising result that every Pythagorean triplet
is the pre product of the triplet $(3, 4, 5)$ presented as a column by a product of three matrices,
that every triplet is obtained in this manner exactly once and in primitive form.

In this paper we show a similar result for integer triangles with an angle of 60 degrees (also known as Eisensteinian triplets). We show that
any such triangle is obtained by pre-multiplication $(7,8,5)$ or $(13,15,7)$ by a product of five matrices. 

The result might have applications in enumerating points with rational distance from the origin on the hexagonal lattice.
\end{abstract}

\maketitle
\

\begin{acknowledgements*}
	We would like to thank Nalinpat Ponoi for pointing out a small mistake in the printed version published in American Mathematical Monthly. The mistake does not appear in this version.
\end{acknowledgements*}

\specialsection*{Introduction}
\label{label:sec}

\begin{definition}
A primitive Eisensteinian triangle is an integer triangle with side lengths $(a,b,c)$ such that the angle opposite the side of length $a$ is of $\frac{\pi}{3}$ and $\gcd(a,b,c) = 1$. Here, we will also follow the convention of ordering the triplet such that $b>c$.
\end{definition}

The original Result by Berggren \cite{berggren} was that any Pythagorean triplet is the left product of the triplet $(3, 4, 5)$ presented as a column with a product of the three matrices:
\begin{equation*}
	\begin{pmatrix} 1 & -2 & 2 \\ 2 & -1 & 2 \\ 2 & -2 & 3 \end{pmatrix},
	\begin{pmatrix} 1 & 2 & 2 \\ 2 & 1 & 2 \\ 2 & 2 & 3 \end{pmatrix},
	\begin{pmatrix} -1 & 2 & 2 \\ -2 & 1 & 2 \\ -2 & 2 & 3 \end{pmatrix}
\end{equation*}

Multiple proof of this were published. We use the outlines of the proof given by Shin Ichi Katayama here \cite{katayama2013modified} to prove the following generalization:

\begin{theorem*}

An integer triangle is primitive Eisensteinian if, and only if it has side lengths $(1,1,1)$ or it has side lengths $(a,b,c)$ or $(a,b,b-c)$ such that:
\begin{equation*}
\begin{pmatrix}a\\
b\\
c
\end{pmatrix}
=M \cdot v \text{, where } 
v \in \left\{ \begin{pmatrix}7\\
8\\
5
\end{pmatrix},\begin{pmatrix}13\\
15\\
7
\end{pmatrix}\right\} 
\end{equation*}
And $M$ is a finite (or empty) product of the matrices:
\[
	M_{1}=
	\begin{pmatrix}7&-6&6\\8&-7&7\\4&-4&3 \end{pmatrix},
	M_{2}=
	\begin{pmatrix}7 & 6 & -6\\8 & 7 & -7\\4 & 3 & -4\end{pmatrix},
	M_{3}=
	\begin{pmatrix}7 & 6 & 0\\8 & 7 & 0\\4 & 3 & 1\end{pmatrix},
\]
\[
	M_{4}=
	\begin{pmatrix}7 & 0 & 6\\8 & 0 & 7\\4 & 1 & 3\end{pmatrix},
	M_{5}=
	\begin{pmatrix}7 & 0 & -6\\8 & 0 & -7\\4 & 1 & -4\end{pmatrix}
\]
Any such triangle is obtained exactly once.

\end{theorem*}

This gives rise to the following forest of primitive eisensteinian triplets where every such triplet appear exactly once:

\begin{figure}[H]
\scalebox{0.5}{
\begin{tikzpicture}[level distance=4cm, level 1/.style={sibling distance=2.5cm}, level 2/.style={sibling distance=0.5cm}]
\node {(13,15,8/7)}[grow=left]
child { node {(49,55,16/39)}
child { node{(109,119,24/95)}}
child { node{(577,665,297/368)}}
child { node{(673,777,377/400)}}
child { node{(439,504,299/205)}}
child { node{(247,280,187/93)}}
}child { node {(133,153,65/88)}
child { node{(403,448,115/333)}}
child { node{(1459,1680,731/949)}}
child { node{(1849,2135,1056/1079)}}
child { node{(1321,1519,880/639)}}
child { node{(541,609,425/184)}}
}child { node {(181,209,105/104)}
child { node{(643,720,203/517)}}
child { node{(1891,2176,931/1245)}}
child { node{(2521,2911,1456/1455)}}
child { node{(1897,2183,1248/935)}}
child { node{(637,713,513/200)}}
}child { node {(139,160,91/69)}
child { node{(559,629,189/440)}}
child { node{(1387,1595,672/923)}}
child { node{(1933,2232,1127/1105)}}
child { node{(1519,1749,989/760)}}
child { node{(427,475,352/123)}}
}child { node {(43,48,35/13)}
child { node{(223,253,85/168)}}
child { node{(379,435,176/259)}}
child { node{(589,680,351/329)}}
child { node{(511,589,325/264)}}
child { node{(91,99,80/19)}}
};
\end{tikzpicture}
}
\scalebox{0.5}{
\begin{tikzpicture}[level distance=4cm, level 1/.style={sibling distance=2.5cm}, level 2/.style={sibling distance=0.5cm}]
\node {(7,8,3/5)}[grow=right]
child { node {(19,21,5/16)}
child { node{(37,40,7/33)}}
child { node{(229,264,119/145)}}
child { node{(259,299,144/155)}}
child { node{(163,187,112/75)}}
child { node{(103,117,77/40)}}
}child { node {(79,91,40/51)}
child { node{(247,275,72/203)}}
child { node{(859,989,429/560)}}
child { node{(1099,1269,629/640)}}
child { node{(793,912,527/385)}}
child { node{(313,352,247/105)}}
}child { node {(97,112,55/57)}
child { node{(337,377,105/272)}}
child { node{(1021,1175,504/671)}}
child { node{(1351,1560,779/781)}}
child { node{(1009,1161,665/496)}}
child { node{(349,391,280/111)}}
}child { node {(67,77,45/32)}
child { node{(277,312,95/217)}}
child { node{(661,760,319/441)}}
child { node{(931,1075,544/531)}}
child { node{(739,851,480/371)}}
child { node{(199,221,165/56)}}
}child { node {(31,35,24/11)}
child { node{(151,171,56/115)}}
child { node{(283,325,133/192)}}
child { node{(427,493,253/240)}}
child { node{(361,416,231/185)}}
child { node{(73,80,63/17)}}
};
\end{tikzpicture}
}
\end{figure}

The tree is defined such that every node is a multiplication of it's parent with one of the matrices defined in the previous theorem.

\section{Integer Triangles with $\frac{\pi}{3}$ angle}

By the cosine theorem we get that a triangle with side lengths $a,b,c$  has a $\frac{\pi}{3}$ angle
opposite to the side with length $a$ if, and only if $(a,b,c)$ satisfy:
\begin{equation}
	a^2 = b^2 + c^2 - bc
\end{equation}

For each Eisensteinian triplet $(a,b,c)$ the triplet $(a,b,b-c)$ is also Eisensteinian.
a geometric proof is shown in the following figure:

\begin{figure}[H]
\includegraphics[scale=0.3]{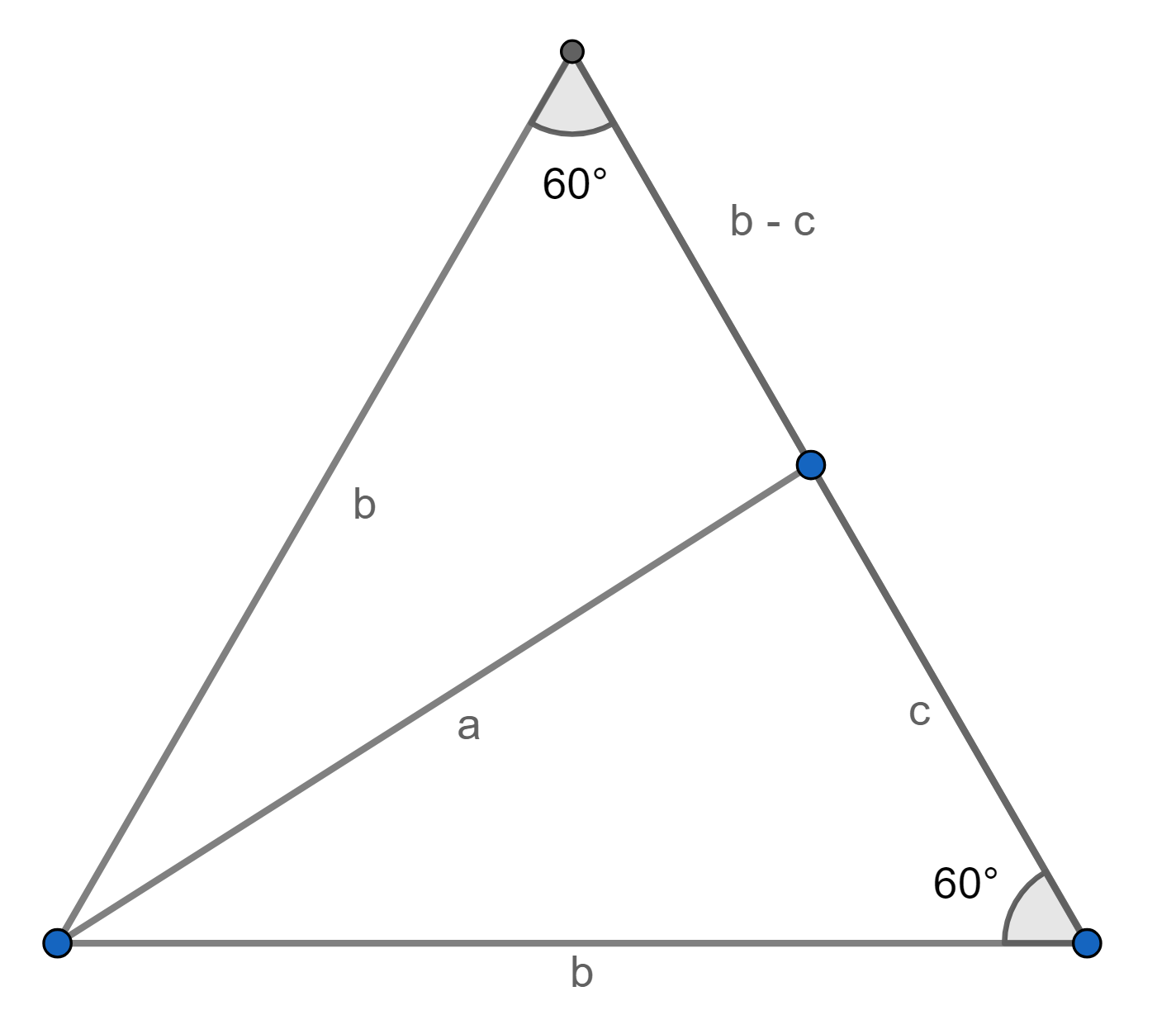}
\caption{the twin Eisensteinian triplet}
\label{fig:twin}
\end{figure}

We will use the following theorem proved here \cite{gilder1982integer}:
\begin{theorem}
	A non-equilateral integer sided triangle is primitive eisensteinian if,
and only if, it is of the class of triangles whose sides are $m^2 + mn + n^2$, $m^2 + 2mn$ and either $n^2 + 2mn$ or $m^2 - n^2$,
where $m$ and $n$ are positive co-prime integers with $m>n>0$ and $m\not\equiv n\left(\mod3\right)$. Any such triangle is achieved exactly once.
\end{theorem}

\section{The Stern–Brocot tree}
By the last theorem, it is sufficient to find and number all couples $(m,n) \in \mathbb{N}^2$ S.T. $m>n>0$, $(m,n)=1$ 
and $m\not\equiv n\left(\mod3\right)$.
This is equivalent to finding all factions $0<\frac{n}{m}<1$ in reduced form which satisfy $m\not\equiv n\left(\mod3\right)$.
To find all such fractions we will look at the Stern-Brocot tree discovered by Moritz Stern \cite{stern1858ueber} and  Achille Brocot  \cite{brocot1861calcul}.

The Stern–Brocot tree is a binary tree of the rational numbers where any vertex is the mediant of his predecessor and successor (in an InOrder ordering) in the tree above it where the predecessor of the root is defined to be $\frac{0}{1}$ and the successor is $\frac{1}{1}$. giving rise to the following tree:

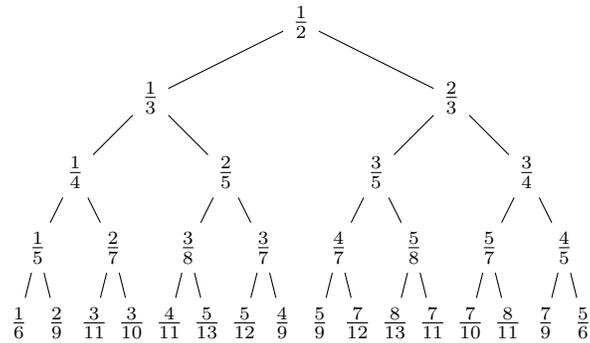
\begin{figure}[H]
\begin{tikzpicture}[level distance=1cm, level 1/.style={sibling distance=4.0cm}, level 2/.style={sibling distance=2.0cm}, level 3/.style={sibling distance=1.0cm}, level 4/.style={sibling distance=0.5cm}]
\node {$\frac{1}{2}$}
child {node {$\frac{1}{3}$}
child {node {$\frac{1}{4}$}
child {node {$\frac{1}{5}$}
child {node {$\frac{1}{6}$}}
child {node {$\frac{2}{9}$}}
}
child {node {$\frac{2}{7}$}
child {node {$\frac{3}{11}$}}
child {node {$\frac{3}{10}$}}
}
}
child {node {$\frac{2}{5}$}
child {node {$\frac{3}{8}$}
child {node {$\frac{4}{11}$}}
child {node {$\frac{5}{13}$}}
}
child {node {$\frac{3}{7}$}
child {node {$\frac{5}{12}$}}
child {node {$\frac{4}{9}$}}
}
}
}
child {node {$\frac{2}{3}$}
child {node {$\frac{3}{5}$}
child {node {$\frac{4}{7}$}
child {node {$\frac{5}{9}$}}
child {node {$\frac{7}{12}$}}
}
child {node {$\frac{5}{8}$}
child {node {$\frac{8}{13}$}}
child {node {$\frac{7}{11}$}}
}
}
child {node {$\frac{3}{4}$}
child {node {$\frac{5}{7}$}
child {node {$\frac{7}{10}$}}
child {node {$\frac{8}{11}$}}
}
child {node {$\frac{4}{5}$}
child {node {$\frac{7}{9}$}}
child {node {$\frac{5}{6}$}}
}
}
};
\end{tikzpicture}
\caption{the Stern-Brocot tree}
\label{fig:tree1}
\end{figure}

\begin{theorem}
	Every rational number between $0$ and $1$ appears exactly once in the Stern-Brocot Tree in reduced form.
\end{theorem}

A nice proof of this is given here \cite{SBTreeCut}

So, we are interested in finding all the entries in the Stern Brocot tree for which the denominator and the numerator are of distinct remainder modulus 3.  omitting the fractions $\frac{n}{m}$ for which $m\equiv n\left(\mod3\right)$,and rearranging the branches length we get the following tree:
\begin{figure}[H]
\begin{tikzpicture}
\node {$\frac{1}{2}$}
child[level distance = 0.5cm, sibling distance = 2.5cm] {node {$\frac{1}{3}$}
child[level distance = 1cm, sibling distance = 5cm] {
child[level distance = 0.5cm, sibling distance = 2.5cm] {node {$\frac{1}{5}$}
	child[level distance = 0.5cm, sibling distance = 0.5cm] {node {$\frac{1}{6}$}
	child[level distance = 1cm, sibling distance = 1cm] {
	child[level distance = 0.5cm, sibling distance = 0.5cm] {node {$\frac{1}{8}$}}
	child[level distance = 0.5cm, sibling distance = 0.5cm] {node {$\frac{2}{13}$}}
	}
	child[level distance = 1cm, sibling distance = 1cm] {
	child[level distance = 0.5cm, sibling distance = 0.5cm] {node {$\frac{3}{17}$}}
	child[level distance = 0.5cm, sibling distance = 0.5cm] {node {$\frac{3}{16}$}}
	}
	}
	child[level distance = 2cm, sibling distance = 2cm] {node{$\frac{2}{9}$}}
}
child[level distance = 0.5cm, sibling distance = 2.5cm] {node {$\frac{2}{7}$}
	child[level distance = 2cm, sibling distance = 2cm]{node{$\frac{3}{11}$}}
	child[level distance = 0.5cm, sibling distance = 0.5cm] {node{$\frac{3}{10}$}
	child[level distance = 1cm, sibling distance = 1cm] {
	child[level distance = 0.5cm, sibling distance = 0.5cm] {node{$\frac{7}{24}$}}
	child[level distance = 0.5cm, sibling distance = 0.5cm] {node{$\frac{8}{27}$}}
	}
	child[level distance = 1cm, sibling distance = 1cm]{
	child[level distance = 0.5cm, sibling distance = 0.5cm] {node{$\frac{7}{23}$}}
	child[level distance = 0.5cm, sibling distance = 0.5cm] {node{$\frac{5}{16}$}}
	}
	}
}
}
child[level distance = 1cm, sibling distance = 5cm] {
child[level distance = 0.5cm, sibling distance = 2.5cm] {node {$\frac{3}{8}$}
	child[level distance = 0.5cm, sibling distance = 0.5cm] {node {$\frac{4}{11}$}
	child[level distance = 1cm, sibling distance = 1cm] {
	child[level distance = 0.5cm, sibling distance = 0.5cm] {node {$\frac{6}{17}$}}
	child[level distance = 0.5cm, sibling distance = 0.5cm] {node {$\frac{9}{25}$}}
	}
	child[level distance = 1cm, sibling distance = 1cm] {
	child[level distance = 0.5cm, sibling distance = 0.5cm] {node {$\frac{11}{30}$}}
	child[level distance = 0.5cm, sibling distance = 0.5cm] {node {$\frac{10}{27}$}}
	}
	}
	child[level distance = 2cm, sibling distance = 2cm] {node{$\frac{5}{13}$}}
}
child[level distance = 0.5cm, sibling distance = 2.5cm] {node {$\frac{3}{7}$}
	child[level distance = 2cm, sibling distance = 2cm]{node{$\frac{5}{12}$}}
	child[level distance = 0.5cm, sibling distance = 0.5cm] {node{$\frac{4}{9}$}
	child[level distance = 1cm, sibling distance = 1cm] {
	child[level distance = 0.5cm, sibling distance = 0.5cm] {node{$\frac{10}{23}$}}
	child[level distance = 0.5cm, sibling distance = 0.5cm] {node{$\frac{11}{25}$}}
	}
	child[level distance = 1cm, sibling distance = 1cm]{
	child[level distance = 0.5cm, sibling distance = 0.5cm] {node{$\frac{9}{20}$}}
	child[level distance = 0.5cm, sibling distance = 0.5cm] {node{$\frac{6}{13}$}}
	}
	}
}
}
}
child[level distance = 2cm, sibling distance = 10cm] {node {$\frac{2}{3}$}
	child[level distance = 0.5cm, sibling distance = 0.5cm] {node {$\frac{3}{5}$}
	child[level distance = 1cm, sibling distance = 1cm] {
	child[level distance = 0.5cm, sibling distance = 0.5cm] {node {$\frac{5}{9}$}}
	child[level distance = 0.5cm, sibling distance = 0.5cm] {node {$\frac{7}{12}$}}
	}
	child[level distance = 1cm, sibling distance = 1cm] {
	child[level distance = 0.5cm, sibling distance = 0.5cm] {node {$\frac{8}{13}$}}
	child[level distance = 0.5cm, sibling distance = 0.5cm] {node {$\frac{7}{11}$}}
	}
	}
	child[level distance = 2cm, sibling distance = 2cm] {node{$\frac{3}{4}$}}
};
\end{tikzpicture}
\caption{the modified Stern-Brocot tree}
\label{fig:tree2}
\end{figure}
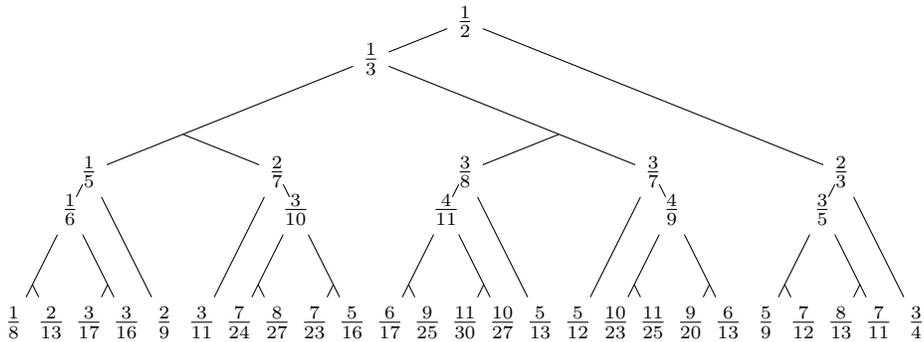

\begin{claim}
	This pattern of ommited fractions (for which the denominator and the numerator are equivalent $\mod 3$) continues forever.
\end{claim}

\begin{proof} 
	when looking on the first occurrence of the patterned tree, S.T. for every fraction $\frac{n}{m}$ we write $m-n \mod 3$ we get:

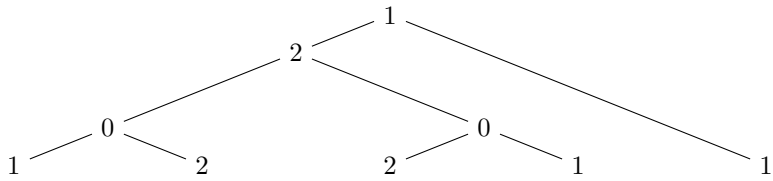
\begin{figure}[H]
\begin{tikzpicture}
\node {$1$}
child[level distance = 0.5cm, sibling distance = 2.5cm] {node {$2$}
child[level distance = 1cm, sibling distance = 5cm] {node{$0$}
child[level distance = 0.5cm, sibling distance = 2.5cm] {node {$1$}}
child[level distance = 0.5cm, sibling distance = 2.5cm] {node {$2$}}
}
child[level distance = 1cm, sibling distance = 5cm] {node{$0$}
child[level distance = 0.5cm, sibling distance = 2.5cm] {node {$2$}}
child[level distance = 0.5cm, sibling distance = 2.5cm] {node {$1$}}
}
}
child[level distance = 2cm, sibling distance = 10cm] {node {$1$}};
\end{tikzpicture}
\caption{the difference $mod 3$ tree}
\label{fig: tree_mod}
\end{figure}
	Now, since each leaf has one of it's predecessor or successor zero and the other the same as his, we get the it's subtree is of the same pattern of difference between enumerator and denominator modulu 3.
\end{proof}

\section{The Modular Tree}
Taking the tree we now wish to understand going through paths in the tree in matrix multiplication terms. We can take each fraction and replace it with the matrix who's entries are the successor and predecessor expressed as column vectors. For example, $\frac{1}{2}$'s matrix will be 
$\begin{pmatrix}0 & 1\\1 & 1\end{pmatrix}$. For entries where the pattern flip in direction in the subtree we will flip the two columns. This will compensate for the change in direction when moving to matrix multiplication, since it is equivalent to flipping the entire subtree.

\begin{figure}[H]
\begin{tikzpicture}
\node {$\begin{pmatrix}0 & 1\\1 & 1\end{pmatrix}$}
child[level distance = 0.5cm, sibling distance = 2.5cm] {node {$\begin{pmatrix}0 & 1\\1 & 2\end{pmatrix}$}
child[level distance = 1cm, sibling distance = 5cm] {
child[level distance = 0.5cm, sibling distance = 2.5cm] {node {$\begin{pmatrix}0 & 1\\1 & 4\end{pmatrix}$}
}
child[level distance = 0.5cm, sibling distance = 2.5cm] {node {$\begin{pmatrix}1 & 1\\3 & 4\end{pmatrix}$}
}
}
child[level distance = 1cm, sibling distance = 5cm] {
child[level distance = 0.5cm, sibling distance = 2.5cm] {node {$\begin{pmatrix}1 & 2\\3 & 5\end{pmatrix}$}
}
child[level distance = 0.5cm, sibling distance = 2.5cm] {node {$\begin{pmatrix}1 & 2\\2 & 5\end{pmatrix}$}
}
}
}
child[level distance = 2cm, sibling distance = 10cm] {node {$\begin{pmatrix}1 & 1\\2 & 1\end{pmatrix}$}
};
\end{tikzpicture}
\caption{the modified Stern-Brocot tree}
\label{fig:treematrices}
\end{figure}
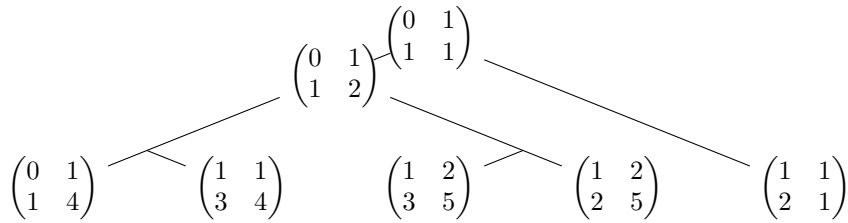
For every $\frac{n}{m}$ fraction in the modified tree we get that if $\begin{pmatrix}a & c\\b & d\end{pmatrix}$ is the related matrix we get:
\[
	\begin{pmatrix}a & c\\b & d\end{pmatrix}\begin{pmatrix}1\\1\end{pmatrix}=\begin{pmatrix}n\\m\end{pmatrix}
\]
If the root of a single occurrence of the pattern is $\begin{pmatrix}a & c\\b & d\end{pmatrix}$ we get that the rest of it's pattern will be as following:
\begin{center}
\begin{tikzpicture}
\node {$\begin{pmatrix}a & c\\b & d\end{pmatrix}$}
child[level distance = 1cm, sibling distance = 3cm] {node {$\begin{pmatrix}a & a+c\\b & b+d\end{pmatrix}$}
child[level distance = 2cm, sibling distance = 6cm] {
child[level distance = 1cm, sibling distance = 3cm] {node {$\begin{pmatrix}a & 3a+c\\b & 3b+d\end{pmatrix}$}
}
child[level distance = 1cm, sibling distance = 3cm] {node {$\begin{pmatrix}2a+c & 3a+c\\2b+d & 3b+d\end{pmatrix}$}
}
}
child[level distance = 2cm, sibling distance = 6cm] {
child[level distance = 1cm, sibling distance = 3cm] {node {$\begin{pmatrix}2a+c & 3a+2c\\2b+d & 3b+2d\end{pmatrix}$}
}
child[level distance = 1cm, sibling distance = 3cm] {node {$\begin{pmatrix}a+c & 3a+2c\\b+d & 3b+2d\end{pmatrix}$}
}
}
}
child[level distance = 4cm, sibling distance = 12cm] {node {$\begin{pmatrix}a+c & c\\b+d & d\end{pmatrix}$}
};
\end{tikzpicture}
\end{center}
Letting:
\begin{align*}
	S'=\begin{pmatrix}1 & 1\\0 & 1\end{pmatrix}, F_1=\begin{pmatrix}1 & 3\\0 & 1\end{pmatrix}, F_2=\begin{pmatrix}2 & 3\\1 & 1\end{pmatrix}\\
	F_3=\begin{pmatrix}2 & 3\\1 & 2\end{pmatrix}, F_4=\begin{pmatrix}1 & 3\\1 & 2\end{pmatrix}, F_5=\begin{pmatrix}1 & 0\\1 & 1\end{pmatrix}\end{align*}
We get:
\begin{align*}
	\begin{pmatrix}a & c\\b & d\end{pmatrix}S'&=\begin{pmatrix}a & a+c\\b & b+d\end{pmatrix}\\
	\begin{pmatrix}a & c\\b & d\end{pmatrix}F_1&=\begin{pmatrix}a & 3a+c\\b & 3b+d\end{pmatrix}\\
	\begin{pmatrix}a & c\\b & d\end{pmatrix}F_2&=\begin{pmatrix}2a+c & 3a+c\\2b+d & 3b+d\end{pmatrix}\\	
	\begin{pmatrix}a & c\\b & d\end{pmatrix}F_3&=\begin{pmatrix}2a+c & 3a+2c\\2b+d & 3b+2d\end{pmatrix}\\
	\begin{pmatrix}a & c\\b & d\end{pmatrix}F_4&=\begin{pmatrix}a+c & 3a+2c\\b+d & 3b+2d\end{pmatrix}\\
	\begin{pmatrix}a & c\\b & d\end{pmatrix}F_5&=\begin{pmatrix}a+c & c\\b+d & d\end{pmatrix}
\end{align*}

Therefore, we get that every fraction $\frac{n}{m}$ such that $m\not\equiv n\left(\mod3\right)$ is of the form:
\begin{multline}
	\begin{pmatrix}0 & 1\\1 & 1\end{pmatrix}\left(\prod_{i=1}^kF_{\alpha_i}\right)(S')^{\varepsilon}\begin{pmatrix}1\\1\end{pmatrix}=\begin{pmatrix}n\\m\end{pmatrix}\\ 
	\text{where }\forall i (1\leq\alpha_i\leq5), \text{ and } \varepsilon \in \{0,1\}
\end{multline}

Let:

\begin{equation*}
	A_i = \begin{pmatrix}0 & 1\\1 & 1\end{pmatrix}F_i\begin{pmatrix}0 & 1\\1 & 1\end{pmatrix}^{-1},\text{  }
	S=\begin{pmatrix}0 & 1\\1 & 1\end{pmatrix}S'\begin{pmatrix}0 & 1\\1 & 1\end{pmatrix}^{-1}
\end{equation*}
 And so,
\begin{multline}
	\left(\prod_{i=1}^kA_{\alpha_i}\right)S^{\varepsilon}\begin{pmatrix}0 & 1\\1 & 1\end{pmatrix}\begin{pmatrix}1\\1\end{pmatrix}=
	\left(\prod_{i=1}^kA_{\alpha_i}\right)S^{\varepsilon}\begin{pmatrix}1\\2\end{pmatrix}=
	\begin{pmatrix}n\\m\end{pmatrix}\\ 
	\text{where }\forall i (1\leq\alpha_i\leq5), \text{ and } \varepsilon \in \{0,1\}
\end{multline}
Where:
\begin{align*}
	S=\begin{pmatrix}1&0\\1&1 \end{pmatrix}, A_1=\begin{pmatrix}1 & 0\\3 & 1\end{pmatrix}, A_2=\begin{pmatrix}0 & 1\\1 & 3\end{pmatrix}\\
	A_3=\begin{pmatrix}1 & 1\\2 & 3\end{pmatrix}, A_4=\begin{pmatrix}1 & 1\\3 & 2\end{pmatrix}, A_5=\begin{pmatrix}0 & 1\\-1 & 2\end{pmatrix}\end{align*}

\section{Proof of the result}

Our first matrix is $S$ which splits our forest to it's two trees, since $S \begin{pmatrix}1\\2\end{pmatrix}=\begin{pmatrix}1\\3\end{pmatrix}$ and the two corresponding triplets are $(7,8,5/3)$ and $(13,15,7/8)$.

If $\frac{n}{m}$ is the fraction related to the two twin triplets $(n^2 + nm + m^2, m^2 +2nm, n^2+2nm/m^2-n^2)$ we have that applying $A_1$ to $\begin{pmatrix}n\\m\end{pmatrix}$ results in:
\begin{align*}
	A_1\frac{n}{m}=\frac{n}{3n+m}&\mapsto 
	\begin{pmatrix}n^2 +n(3n+m)+ (3n+m)^2  \\2n(3n+m)+(3n+m)^2\\ n^2+2n(3n+m)  \end{pmatrix}
	&=\begin{pmatrix}13n^2+7nm+m^2\\15n^2+8nm+m^2\\7n^2+2nm \end{pmatrix}
	\\&=\begin{pmatrix}7&-6&6\\8&-7&7\\4&-4&3 \end{pmatrix}
	\begin{pmatrix}n^2 +nm+ m^2  \\2nm+m^2\\ n^2+2nm\end{pmatrix}
	&=M_1 \begin{pmatrix}n^2 +nm+ m^2  \\2nm+m^2\\ n^2+2nm\end{pmatrix}
\end{align*}
So, in the same manner, multiplying the $m,n$ vector by $A_i$ is equivalent to multiplying the resulting triplet by $M_i$. Since we already established that the tree spans all $m,n$ vectors where $n\not\equiv m \mod 3$ and the all triplets are the result of inserting such $m,n$ to the formula and that every twin triplet is achieved once in primitive form, we get our result, as stated in the introduction \ref{label:sec}.

\bibliographystyle{plain}
\bibliography{bibl}

\end{document}